\documentclass[11pt]{amsart}
\usepackage[utf8]{inputenc}
\usepackage{graphicx}
\usepackage{amssymb}
\usepackage{amscd}
\usepackage{hyperref}
\usepackage{amsmath,amsthm}
\usepackage{latexsym}
\usepackage{caption}
\usepackage{color}
\usepackage{cite}

\newtheorem{thm}{Theorem}
\newtheorem{lem}[thm]{Lemma}
\newtheorem{cor}[thm]{Corollary}
\newtheorem{defn}{Definition}
\newtheorem{rem}{Remark}
\newtheorem{prop}[thm]{Proposition}
\newtheorem{exa}{Example}

\def\Z{\mathbb{Z}}

\newcommand{\tr}{\operatorname{tr}}

\begin{document}
\title[$D(N)$-quadruples in upper-triangular $2\times2$ integer matrices]{$D(N)$-quadruples in upper-triangular $2\times2$ integer matrices}
\author[Andrej Dujella, Zrinka Franu\v{s}i\'{c}]{Andrej Dujella, Zrinka Franu\v{s}i\'{c}}

\address{
University of Zagreb Faculty of Science, Department of Mathematics, 
Bijeni{\v c}ka cesta 30, 10000 Zagreb, Croatia}

\email{andrej.dujella@math.pmf.unizg.hr, zrinka.franusic@math.pmf.unizg.hr}

\keywords{Diophantine $m$-tuples, Jordan product, upper-triangular
integer matrices, difference of two squares, Diophantine
equations}

\subjclass{Primary 11C20, 11D99; Secondary 15B36, 17C50} 


\begin{abstract}
We introduce analogues of Diophantine $D(N)$-$m$-tuples in the noncommutative ring $M_2(\mathbb Z)$ of $2\times2$ integer matrices. Besides definitions based on the standard matrix product, we consider a symmetric version defined via the Jordan product
$$A\circ B=\frac12(AB+BA).$$
Special attention is devoted to upper-triangular integer matrices $UT_2(\mathbb Z)$, where squares admit a particularly simple description. Motivated by the classical connection between representations of $n$ as a difference of two squares and the existence of $D(n)$-quadruples in commutative rings, we investigate the existence of Jordan $D(N)$-quadruples in $UT_2(\mathbb Z)$.
\end{abstract}

\maketitle

\section{Introduction and motivation}
As the supplement \cite{duje-open} to the book \cite{duje-ell-cu}, a list of open problems was presented. In this paper, we address the specific problem formulated as follows:

\textit{Define an analogue of a $D(n)$-$m$-tuple in a noncommutative ring and prove some non-trivial results concerning the existence of $D(n)$-quadruples, e.g. in the ring of $2 \times 2$ matrices with integer entries.}

Recall that \textbf{$D(n)$-$m$-tuples} are traditionally defined in a commutative ring $\mathcal R$ with  unity as a set of $m$ distinct, nonzero elements $\{a_{1}, \ldots ,a_{m}\}$ in  $\mathcal R$  such that the product of any two distinct elements, increased by $n$,  is a perfect square in $\mathcal R$, i.e. $a_ia_j+n=x_{ij}^2$ for all $1\leq
i<j\leq m$, where $n\in\mathcal R\backslash \{0\}$. If equal elements and zero are allowed, {\color{black} we refer to such a family (or $m$-tuple) as having the \textbf{$D(n)$-property}. Also, we allow $n=0$ when referring only to the $D(n)$-property.}

So far $D(n)$-$m$-tuples have been extensively investigated in various rings of integers of number fields, such as $\Z,\Z[\sqrt{d}],\Z[(1+\sqrt{d})/2]$, as well as in rings of polynomials with integer coefficients. Interestingly, in many of these rings, an equivalence has been established between the representability of $n$ as a difference of two squares and the existence of a $D(n)$-quadruple (see \cite{duje-acta93,duje-glas97,zr-comm,zr-hung,zr-rama,zr-mis,zr-borka,zr-soldo,az,ljerka}). Yet, this ``rule'' is not without its counterexamples. In the rings of integers of certain quadratic fields, there exist elements $n$ that cannot be represented as a difference of two squares, yet $D(n)$-quadruples still exist (\!\!\cite{cgh,cgh2,gup}). 

A classical construction shows that every $D(n)$-pair $\{a,b\}$ satisfying $ab+n=r^2$ can be extended to a \textbf{regular  $D(n)$-triple} 
$$\{a, b, c_{\pm}\},~~  c_{\pm}=a+b\pm  2r $$
provided that the chosen element $c_{\pm}$
does not belong to $\{0,a,b\}$. Following this construction, one can consider sets of the form:
 $$\{a, b, a + b \pm 2r, a + 4b \pm 4r\}.$$
 These sets are $D(n)$-quadruples if the condition $a(a + 4b + 4r)+n=\square$ is satisfied and if the elements are nonzero and distinct. Building on this approach, Dujella \cite{duje-graz96} derived several parametric formulas for $D(n)$-quadruples. For instance, the set
 \begin{equation}\label{poli.formula}
     \{{\ell, \ell(3k + 1)^2 + 2k, \ell(3k + 2)^2 + 2k + 2, 9\ell(2k + 1)^2 + 8k + 4}\}
 \end{equation}
is shown to be a $D(n)$-quadruple with $n = 2\ell(2k + 1) + 1$, assuming the elements are mutually distinct and nonzero. These polynomial families illustrate that for infinitely many values of $n$, $D(n)$-quadruples exist in abundance within commutative rings.

 Furthermore, an important “scaling property” of these sets is as follows: if the set $\{a_1, \dots, a_m\}$ is a $D(n)$-$m$-tuple in a commutative integral domain $\mathcal{R}$, then for any $w \in \mathcal{R} \setminus \{0\}$, the set$$\{a_1w, \dots, a_mw\} \text{ is a } D(nw^2)\text{-}m\text{-tuple}.$$

 The aim of the present paper is to initiate the study of Diophantine $D(N)$-tuples in a genuinely noncommutative setting. We introduce several natural analogues of $D(N)$-tuples in $M_2(\Z)$, with particular emphasis on the symmetric definition induced by the Jordan product. Our main results concern the ring $UT_2(\Z)$ of upper-triangular $2\times2$ integer matrices. Using the characterization of upper-triangular matrices representable as differences of two squares obtained in \cite{diz2}, we prove that every matrix $N\in UT_2(\Z)$ representable as a difference of two squares in $UT_2(\Z)$ admits infinitely many Jordan $D(N)$-quadruples (Theorem \ref{thm:quadruple_for_diff_sq} and Corollary \ref{cor:ex-of-quadruples}). Thus, to the best of our knowledge, $UT_2(\Z)$ provides the first structure in which the implication
$$N=K^2-A^2 \quad \Longrightarrow \quad
\text{there exists a Jordan }D(N)\text{-quadruple}$$
can be established by combining a structural characterization of matrices representable as differences of two squares with an explicit construction of Jordan $D(N)$-quadruples.
 
 This approach differs substantially from earlier investigations of Diophantine quadruples in rings of algebraic integers (see \cite{zr-comm,zr-rama,zr-hung,zr-mis,zr-soldo,zr-borka}), where existence was typically obtained through explicit polynomial constructions (using  \eqref{poli.formula}, for example). We also investigate matrices $N$ that are not representable as differences of two squares. In several infinite families we prove that Jordan $D(N)$-quadruples do not exist by means of modular obstructions (Propositions \ref{prop:diag-2mod4}, \ref{prop:4k3-4n3-odd}, \ref{prop:8k1-8n5-odd}, \ref{prop:8k5-8n5-odd}, 
 \ref{prop:16k12-16n12-m},  \ref{prop:16k4-odd-16m12} and \ref{prop:16k8-odd-16m8}), while Propositions \ref{prop:quadruple-1m1}, \ref{prop:quintuple}, \ref{prop:s2-m-t2} and \ref{prop:4-2t-4} provide explicit constructions of Jordan $D(N)$-quadruples for matrices belonging to exceptional congruence classes. {\color{black} In particular, we show that if both diagonal entries of $N$ are integer squares, then a Jordan $D(N)$-quadruple exists for every upper-right entry (Proposition~\ref{prop:s2-m-t2}).} These results demonstrate that the relationship between difference-of-squares representations and the existence of Jordan $D(N)$-quadruples in matrix rings might be more subtle than in the classical commutative setting.

 \section{Matrix analogues of $D(N)$-$m$-tuples}

 Let $M_2(\mathbb{Z})$ be the ring of all $2 \times 2$ matrices with integer entries. It is well known that $M_2(\mathbb{Z})$ forms a noncommutative ring with unity $I$. Due to the lack of commutativity, the standard definition of a Diophantine $m$-tuple must be adapted. We propose two primary frameworks for defining a $D(N)$-$m$-tuple in this context, where $N \in M_2(\mathbb{Z})$.
 
 \subsection{Standard Product Definitions}
 
 The first approach considers the order of multiplication directly.
 \begin{defn}\label{defn:standard} Let $N \in M_2(\mathbb{Z})$.

 An ordered family $(A_1,  \ldots, A_m)$ of distinct nonzero
matrices in $M_2(\mathbb{Z})$
 is called:\begin{itemize}\item a \textbf{$D(N)$ matrix $m$-tuple in ascending order} if for all $1 \le i < j \le m$, there exists $X_{ij} \in M_2(\mathbb{Z})$ such that $A_iA_j + N = X_{ij}^2$,
 \item a \textbf{$D(N)$ matrix $m$-tuple in descending order} if for all $1 \le i < j \le m$, there exists $X_{ji} \in M_2(\mathbb{Z})$ such that $A_jA_i + N = X_{ji}^2$,\item a \textbf{$D(N)$ matrix $m$-tuple} if it is a $D(N)$ matrix $m$-tuple in both ascending and descending order.\end{itemize}\end{defn}
 Note that if an $m$-tuple $(A_1, \ldots, A_m)$ consists of matrices that pairwise commute, these three definitions collapse into the classical definition.

 \begin{exa}
     Let $a\in\Z$ and 
     $$  A=\begin{pmatrix}a + 1& a + 2\\
a &a + 1\end{pmatrix},~ B=\begin{pmatrix}-a - 1 &a + 1\\
a &-a\end{pmatrix},~ C=\begin{pmatrix} 0 &2a + 3\\
0 &-2a - 1\end{pmatrix}.$$
Note that 
$$ AB+I=\begin{pmatrix}0& 1\\
0&1 \end{pmatrix}^2,~ BA+I=\begin{pmatrix}a &a + 1\\
-a& -a - 1 \end{pmatrix}^2,$$
$$ AC+I=\begin{pmatrix}1&1\\0&0 \end{pmatrix}^2,~ CA+I=\begin{pmatrix}
(a + 1)(2a + 1)& 2a^2 + 5a + 3\\
-a(2a + 1)& -a(2a + 3)\end{pmatrix}^2,$$
$$ BC+I=\begin{pmatrix} 1 &-2(a + 1)\\
0& 2a + 1 \end{pmatrix}^2,~ CB+I=\frac{1}{4(1 + a)^2}\begin{pmatrix}2a^2 + 5a + 2& -a(2a + 3)\\ -a(2a + 1)& 2a^2 + 3a + 2\end{pmatrix}^2,$$
where $a\neq-1$.  For $a = -2, 0$  $$CB+I=\begin{pmatrix}
    0 &1\\3 &-2
\end{pmatrix}^2,\quad \begin{pmatrix}
    1 &0\\0 &1
\end{pmatrix},$$ respectively.

Thus, for $a \in \{-2, 0\}$, $(A, B, C)$ is a $D(I)$-triple in the sense of the standard product (in both ascending and descending orders), despite the fact that the matrices do not commute.
 \end{exa}
 
 \subsection{The Jordan Product Approach}
 As an alternative, we can utilize the symmetric part of the matrix product to restore commutativity at the operational level. We define the Jordan product of two matrices $A$ and $B$:
 $$A \circ B = \frac{1}{2}(AB + BA),$$
 see, for example, \cite{spring}. If $\mathbb F$ is a field of characteristic different from \(2\), then  $(M_2(\mathbb{F}), \circ)$ is a Jordan algebra. This operation is commutative but generally non-associative, satisfying the Jordan identity: $(A \circ B) \circ A^2 = A \circ (B \circ A^2)$. Using this operation, we can provide a definition that is inherently symmetric:
 
 \begin{defn}
 The set $\{A_1, A_2, \ldots, A_m\}$ of nonzero distinct matrices in $M_2(\mathbb{Z})$ is a \textbf{$D(N)$ matrix $m$-tuple under the Jordan product} if for all $1 \le i < j \le m$, there exists $X_{ij} \in M_2(\mathbb{Z})$ such that:
 \begin{equation}\label{jordan.property}
     A_i \circ A_j + N = X_{ij}^2.
 \end{equation}
 Throughout the paper we refer to these simply as \textbf{Jordan $D(N)$-$m$-tuples}.
 
A family of matrices $(A_1, A_2, \ldots, A_m)$ in $M_2(\mathbb{Z})$, whose elements need not be distinct {\color{black}or nonzero}, is said to have the \textbf{Jordan $D(N)$-property} if it satisfies \eqref{jordan.property} for all $1 \le i < j \le m$.
 \end{defn}

 Notice that the ordering is essential in Definition \ref{defn:standard}, where ascending and descending matrix $D(N)$-$m$-tuples are distinguished, whereas the Jordan $D(N)$-property is independent of the ordering of the matrices.

 Although the Jordan product ensures symmetry in the $D(N)$-$m$-tuple definition, a potential issue arises as $A \circ B$ may not have integer entries for all pairs $A, B \in M_2(\mathbb{Z})$.  Despite this constraint, the Jordan product offers a significant advantage regarding the extension of $D(N)$-tuples. In this setting, the classical extension formula for triples remains valid. If $\{A, B\}$ is a $D(N)$ pair such that $A \circ B + N = R^2$, then the set
 $$\{A, B, A + B \pm 2R\}$$
 has the Jordan $D(N)$-property. Indeed, bilinearity of the Jordan product allows for the following:
 $$A \circ (A + B \pm 2R) + N = A^2 + A \circ B \pm 2(A \circ R) + N = A^2 + R^2 \pm 2(A \circ R) = (A \pm R)^2.$$

 Unlike in commutative integral domains, the ring $M_2(\mathbb{Z})$ contains zero divisors which allow for the construction of $D(N)$-$m$-tuples of arbitrary length. 

 \begin{exa}\label{exa:nilpotent} 
 Let 
 $$A_i=\begin{pmatrix} 0 & k_i \\ 0 & 0 \end{pmatrix},\ k_i\in\Z.$$
 Note that $A_iA_j=A_jA_i=A_i\circ A_j=0$ (zero matrix), so $A_iA_j+N^2=N^2$ for any $N\in M_2(\Z)$. 
 {\color{black}Thus, by choosing pairwise distinct nonzero integers $k_i$, one obtains $D(N^2)$-$m$-tuples of \textbf{arbitrary length} with respect to both the standard and Jordan products.}
  \end{exa}

 Considering Example \ref{exa:nilpotent}, natural questions arise. How should such degenerate cases be excluded? Should one restrict attention to matrices with nonzero determinant, or require that the product (or Jordan product) of any two matrices in a $D(N)$-tuple be nonzero? 
 
 Indeed, Example \ref{exa:nilpotent} shows that the existence problem becomes trivial if no additional restrictions are imposed. Therefore, throughout the paper we distinguish such degenerate examples from non-degenerate constructions, in which all pairwise Jordan products are nonzero.

 \begin{defn}
     A Jordan $D(N)$-$m$-tuple is called \textbf{non-degenerate} if
$A_i\circ A_j\not=0$, for all $i\not =j$. Otherwise, it is called \textbf{degenerate}.
 \end{defn}

 \begin{lem}
     Let $\{A_1, \dots, A_m\}$ be a Jordan $D(N)$-tuple in $M_2(\Z)$. If $E \in GL_2(\mathbb{Z})$ is a unimodular matrix, then $\{EA_1E^{-1}, \ldots,EA_mE^{-1}\}$ is  a Jordan $D(ENE^{-1})$-$m$-tuple in $M_2(\Z)$.
 \end{lem}
 \begin{proof}
     For every $1 \le i< j \le m$, we have $$A_i \circ A_j + N = X_{ij}^2, ~X_{ij}\in M_2(\Z). $$
     By multiplying the previous relation by $E$ and $ E^{-1}$, we get 
    $$  \frac12(E(A_iA_j+A_jA_i)E^{-1}) + ENE^{-1}=E X_{ij}^2E^{-1},$$
   $$  \frac12((EA_iE^{-1})(EA_jE^{-1})+(EA_jE^{-1})(EA_iE^{-1})) + ENE^{-1}=(E X_{ij}E^{-1})^2.$$ 
   Therefore, 
   $$ (EA_iE^{-1}) \circ (EA_jE^{-1}) + ENE^{-1} = (E X_{ij}E^{-1})^2 .$$
   Because $E$ and $E^{-1}$ are unimodular integer matrices, all these matrix products, including $(EA_iE^{-1}) \circ (EA_jE^{-1})$, have integer entries. Also, conjugation preserves distinctness and nonzeroness.
 \end{proof}

 \begin{exa}
    Conjugating the quadruple from Example \ref{exa.duje} by the unimodular  parametric matrix:
     $$E=\begin{pmatrix}
         1& k\\ m& 1 + km
     \end{pmatrix},~k,m\in\Z,$$
     we obtain a Jordan $D(N_m)$-quadruple for  $N_m=\begin{pmatrix}
         2 (2 - m)& 2\\-2 m^2& 2 (2 + m)
     \end{pmatrix}$:
     $$ \{\begin{pmatrix}
         -m& 1\\-m^2& m
     \end{pmatrix},\begin{pmatrix}
          -2 (1 + 6 k) m& 2 (1 + 6 k)\\-2 m (6 + m + 6 k m)& 2 (6 + m + 6 k m)    \end{pmatrix},$$ $$\begin{pmatrix}
              3 - 2 m + 2 k m& 2 (1 - k)\\2 m (1 - m + k m)& 1 + 2 m - 2 k m
          \end{pmatrix},\begin{pmatrix}
              7 - m + 2 k m& 1 - 2 k\\m (2 - m + 2 k m)& 5 + m - 2 k m
          \end{pmatrix}\}$$
 \end{exa}

 \section{Upper-triangular Jordan $D(N)$-quadruples}

 \subsection{The $UT_2(\mathbb{Z})$ setting}

 A natural setting in which to study the existence of matrix $D(N)$-$m$-tuples is the ring of upper-triangular $2\times2$ integer matrices of the form
$$T=\begin{pmatrix} a & b \\ 0 & c \end{pmatrix},~a,b,c\in\Z.$$
We denote this set by  $UT_2(\mathbb{Z})$. Note that $UT_2(\mathbb{Z})$ forms a ring under standard addition and matrix multiplication. 

The Jordan product is not closed in $UT_2(\mathbb{Z})$ since
$$ \begin{pmatrix} a & b \\ 0 & c \end{pmatrix}\circ \begin{pmatrix} d & e \\ 0 & f \end{pmatrix}=\begin{pmatrix} ad & \frac12(e(a+c)+b(d+f)) \\ 0 & cf \end{pmatrix}.$$
Observe that the structure $(UT_2(\mathbb{Q}),+,\circ)$ forms a commutative but non-associative algebra, known as a \textit{special Jordan algebra}. Despite this,  the set $UT_2(\mathbb{Z})$ is a convenient setting for constructing explicit examples of Jordan $D(N)$-$m$-tuples, because  squares in this ring have a very simple form:
\begin{equation}\label{kvadrat}
    T^2=\begin{pmatrix} a^2 & b(a+c) \\ 0 & c^2 \end{pmatrix}.
\end{equation}

From now on, $UT_2(\Z)$ is our ambient structure. In addition to the fact that the Jordan product need not belong to $UT_2(\Z)$, another distinction is important: an upper-triangular integer matrix that is a square in $M_2(\Z)$ need not have a square root in $UT_2(\Z)$. For instance, 
$$\begin{pmatrix}
    0& 2\\ 1 & 0\end{pmatrix}^2=2I\in UT_2(\Z).$$
whereas $2I$ has no square root in $UT_2(\Z)$. Therefore, in the upper-triangular setting, we require the square roots in the defining relations to belong to $UT_2(\Z)$. This motivates the following definition.
\begin{defn}
    Let $N \in UT_2(\Z)$. A family $\{A_1, \ldots ,A_m\} \subseteq UT_2(\Z)$ has the \textbf{upper-triangular Jordan $D(N)$-property} or \textbf{Jordan $D(N)$-property in $UT_2(\Z)$} if, for every $1 \le i < j \le m$, there exists $X_{ij} \in UT_2(\Z)$ such that \eqref{jordan.property} holds. 

    If the $A_i$ are distinct and nonzero, the family is called an \textbf{upper-triangular Jordan $D(N)$-$m$-tuple} or \textbf{Jordan $D(N)$-$m$-tuple in $UT_2(\Z)$}.
\end{defn}

Throughout Sections 3–5, all Jordan $D(N)$-properties and tuples
are understood in this upper-triangular sense.

 The following lemma shows that scaling all upper-right entries of a Jordan $D(N)$-quadruple by the same nonzero integer scales the upper-right entry of N by the same factor.

 \begin{lem}\label{lem:scale-m}
Let $N_r=\begin{pmatrix}u&r\\0&v\end{pmatrix}\in UT_2(\Z)$
and assume that $\{A_1,A_2,A_3,A_4\}$ is an upper-triangular Jordan $D(N_r)$-quadruple in $UT_2(\Z)$, where
$$A_i=\begin{pmatrix}a_i&b_i\\0&c_i\end{pmatrix},
\qquad 1\le i\le 4.$$
Then for every $m\in\Z\backslash\{0\}$, the matrices
$$A_i^{(m)}=\begin{pmatrix}a_i&m b_i\\0&c_i\end{pmatrix},
\qquad 1\le i\le 4$$
form an upper-triangular Jordan $D(N_{mr})$-quadruple for
$N_{mr}
=\begin{pmatrix}u&mr\\0&v\end{pmatrix}.$
\end{lem}
\begin{proof}
Assume that 
$$A_i\circ A_j+N_r=R_{ij}^2,\qquad R_{ij}=\begin{pmatrix}r_{ij}&z_{ij}\\0&t_{ij}\end{pmatrix},$$
for each pair $1\le i< j\le 4$. Since 
$$ [A_i^{(m)}\circ A_j^{(m)}+N_{mr}]_{11}=r_{ij}^2,\ \  [A_i^{(m)}\circ A_j^{(m)}+N_{mr}]_{22}=t_{ij}^2$$
and
$$ [A_i^{(m)}\circ A_j^{(m)}+N_{mr}]_{12}=m[A_i\circ A_j+N_r]_{12}=mz_{ij}(r_{ij}+t_{ij}),$$
i.e. since the upper-right entry in the Jordan product and in the square of an upper-triangular matrix depends linearly on the upper-right entries when the diagonals are fixed, we obtain
$$ A_i^{(m)}\circ A_j^{(m)}+N_{mr}=\left(R_{ij}^{(m)}\right)^2,\qquad R_{ij}^{(m)}=\begin{pmatrix}r_{ij}&m z_{ij}\\0&t_{ij}\end{pmatrix}.$$
 Hence $\{A_1^{(m)},A_2^{(m)},A_3^{(m)},A_4^{(m)}\}$ is a Jordan $D(N_{mr})$-quadruple.
\end{proof}

 \begin{exa}\label{exa.duje}
     Let $N=\begin{pmatrix}
         4&2\\0&4
     \end{pmatrix}$.  The set
     $$ \{A=\begin{pmatrix}
         0&1\\0&0
     \end{pmatrix}, B=\begin{pmatrix}
         0&2\\0&12
     \end{pmatrix},C=\begin{pmatrix}
         3&2\\0&1
     \end{pmatrix}, D=\begin{pmatrix}
         7&1\\0&5
     \end{pmatrix} \}$$
     is a Jordan $D(N)$-quadruple. Indeed,
     $$ A\circ B+N=\begin{pmatrix}
         2&2\\0&2
     \end{pmatrix}^2, A\circ C+N=\begin{pmatrix}
        2 &1\\0&2
     \end{pmatrix}^2, 
     A\circ D+N=\begin{pmatrix}
         2&2\\0&2
     \end{pmatrix}^2,$$ 
     $$B\circ C+N=\begin{pmatrix}
         2&3\\0&4
     \end{pmatrix}^2,
     B\circ D+N=\begin{pmatrix}
        2&2\\0&8
     \end{pmatrix}^2,
     C\circ D+N=\begin{pmatrix}
         5&2\\0&3
     \end{pmatrix}^2.$$
     Although $N$ is not a difference of two squares in $UT_2(\Z)$ (according to \cite{diz2}), it has several such representations in $M_2(\Z)$:
     
      {\small
     $ N=\begin{pmatrix}
         0 & 2 \\
 2 & 1  \end{pmatrix}^2- \begin{pmatrix}
      0 & 0 \\
 2 & 1     \end{pmatrix}^2=\begin{pmatrix}
       0 & 3 \\
 2 & 1  \end{pmatrix}^2-\begin{pmatrix}
   0 & 1 \\
 2 & 1      \end{pmatrix}^2=\begin{pmatrix}
     1 & 2 \\
 2 & 0
 \end{pmatrix}^2-\begin{pmatrix}
     1 & 0 \\
 2 & 0
 \end{pmatrix}^2$

 $=\begin{pmatrix}
     1 & 3 \\
 2 & 0
 \end{pmatrix}^2-\begin{pmatrix}
     1 & 1 \\
 2 & 0
 \end{pmatrix}^2=\begin{pmatrix}
     3 & 1 \\
 2 & 3
 \end{pmatrix}^2-\begin{pmatrix}
     2 & 1 \\
 3 & 2
 \end{pmatrix}^2.$
 }
 \end{exa}

 This example shows that the existence of Jordan $D(N)$-quadruples does not imply representability of $N$ as a difference of two squares in $UT_2(\Z)$. 

\subsection{Matrices representable as differences of two squares}
In many commutative rings, it has been shown that a $D(n)$-quadruple exists if and only if $n$ can be represented as a difference of two squares, up to a finite number of exceptions. Although we do not expect such an equivalence in matrix structures, we believe that it makes sense to explore the connection between these two concepts.
In \cite{diz2} we gave a complete characterization of upper-triangular matrices that can be represented as a difference of two squares in terms of congruences modulo 16.

We first show that every difference-of-squares representation naturally yields a Jordan $D(N)$-quadruple. Let us start with a  \textit{degenerate construction} of a Jordan $D(N)$-quadruple.

 \begin{lem}[Degenerate construction]
 Let $K,A_0\in M_2(\Z)$ and $N=K^2-A_0^2$. The family
 $$(A_0,A_0,2A_0+2K,5A_0+4K)$$
 has the Jordan $D(N)$-property.    
 \end{lem}
\begin{proof}
The following two families $(A_0,A_0,2A_0+2K)$ and $(A_0,2A_0+2K,5A_0+4K)$ arise from the regular extension identities and have the $D(N)$-property. Indeed, 
    $$ A_0\circ A_0+N=K^2~ \text{ and } ~ A_0+A_0+2K=2A_0+2K,$$
   $$ A_0\circ(2A_0+2K)+N=2A_0^2+2A_0\circ K+K^2-A_0^2=(A_0+K)^2 $$
   and 
   $$A_0+(2A_0+2K)+2(A_0+K)=5A_0+4K.$$
\end{proof}

We next show how a degenerate quadruple arising from a difference of squares representation can be converted into infinitely many proper Jordan
$D(N)$-quadruples.
 This result requires the following technical lemma. 
\begin{lem}\label{lem:tech}
    If $N\in UT_2(\Z)$ can be represented as a difference of two squares in $UT_2(\Z)$, then there exist $A_0,K\in  UT_2(\Z)$ such that 
    \begin{itemize}
        \item[(i)] $N=K^2-A_0^2$,
        \item[(ii)] $3A_0+2K\not=0$, $A_0+K\not=0$ and $5A_0+4K\not=0$,
        \item[(iii)] either $\tr(2A_0+K)\not=0$, or $\tr(A_0)=\tr(K)=0$.
    \end{itemize}
    \end{lem}
\begin{proof}
Assume first that $N=0$. Then we may take $A_0=K=I$, so all three conditions are clearly satisfied.

   Now let $N\not=0$, and suppose that $N=K^2-A_0^2$ for some $A_0,K\in  UT_2(\Z)$. Note that $A_0+K\neq0$ and $-A_0+K\neq0$,
since either equality would imply $K^2=A_0^2$ and hence $N=0$.
We distinguish two cases: $A_0=0$ and $A_0\not =0$.

\textbf{Case} $\boxed{A_0=0}$: Note that $K\not=0$, since $N\not=0$. Hence, $3A_0+2K\not=0$ and $5A_0+4K\not=0$. Moreover, either  $\tr K\not=0$, in which case $\tr(2A_0+K)\not=0$, or $\tr K=0$, in which case $\tr(A_0)=\tr(K)=0$. 

\textbf{Case} $\boxed{A_0\not=0}$: We analyze the possible obstructions to conditions (ii) and (iii) and, if necessary, fix them by replacing $A_0$ by $-A_0$.
    \begin{itemize}  
        \item[$\rhd$]   If $3A_0+2K=0$, then $-3A_0+2K\not=0$, since if both equalities hold, then $A_0=K=0$. In addition, $-5A_0+4K\not=0$, because  $3A_0+2K=0$ and $-5A_0+4K=0$ yield $A_0=K=0$.  Hence after  replacing $A_0$ by $-A_0$, conditions (i) and (ii) are satisfied.

        If $\tr(-2A_0+K)=0$, then together with $3A_0+2K=0$ it follows that $\tr(-A_0)=\tr(K)=0$. Hence, by replacing $A_0$ by $-A_0$, we meet all three conditions. 
  \item[$\rhd$]   If $5A_0+4K=0$, by an analogous argument as in the previous case, all three conditions are satisfied by replacing $A_0$ by $-A_0$.
 \item[$\rhd$]   If $\tr(2A_0+K)=0$ and $\tr A_0\not =0$, then  $\tr K\not=0$. Note that $\tr(-2A_0+K)\not =0$, since otherwise both $\tr(2A_0+K)=0$ and  $\tr(-2A_0+K) =0$  would hold, implying $\tr A_0=\tr K=0$. Also, if $-3A_0+2K=0$, then $\tr(2A_0+K)=0$ implies $\tr A_0=\tr K=0$. Hence, $-3A_0+2K\not =0$. By the same argument we get $-5A_0+4K\not=0$. Therefore, replacing $A_0$ by $-A_0$ satisfies both (ii) and (iii).

\end{itemize}
\end{proof}

 \begin{thm}[Prototype quadruple]\label{thm:quadruple_for_diff_sq}
     Let $A_0,K\in UT_2(\Z)$ be chosen as in Lemma~\ref{lem:tech}, and let 
     $$N=K^2-A_0^2, \qquad  U_t=\begin{pmatrix}0&t\\0&0\end{pmatrix},\ t\in\Z.$$
      Then there exist infinitely many $t\in\Z\backslash\{0\}$ such that 
      \begin{equation} \label{proto_quad}
          \{A_0+U_t,A_0-U_t,2A_0+2K,5A_0+4K\}
      \end{equation}
     is a Jordan $D(N)$-quadruple in $ UT_2(\Z)$.
 \end{thm}
 \begin{proof} First note that $(2A_0+2K)\circ(5A_0+4K)+N=\square$, since $(A_0,2A_0+2K,5A_0+4K)$ has the Jordan $D(N)$-property. Moreover, $U_t$ is a nilpotent matrix, i.e. $U_t^2=0$. Thus, it remains to verify the five remaining pairwise Jordan products.
 \begin{enumerate}
     \item[(i)] $(A_0+U_t)\circ(A_0-U_t)+N=A_0^2+N=K^2$
     \item[(ii)] $(A_0+U_t)\circ(2A_0+2K)+N=(A_0+K+U_t)^2$
       \item[(iii)] $(A_0+U_t)\circ(5A_0+4K)+N=(2A_0+K+U_t)^2+\frac12\begin{pmatrix}
           0& t\cdot \operatorname{tr}( A_0+2K) \\ 0&0
       \end{pmatrix}$
      \item[(iv)] $(A_0-U_t)\circ(2A_0+2K)+N=(A_0+K-U_t)^2$
         \item[(v)] $(A_0-U_t)\circ(5A_0+4K)+N=(2A_0+K-U_t)^2-\frac12\begin{pmatrix}
           0& t\cdot \operatorname{tr} (A_0+2K) \\ 0&0
       \end{pmatrix}$
 \end{enumerate}
Identities (i), (ii), and (iv) follow immediately. It remains to consider
(iii) and (v), i.e. to find $t$ such that the matrices on the right-hand
sides are perfect squares. By Lemma~\ref{lem:tech}, either
$\tr(2A_0+K)\neq0$, or $\tr(A_0)=\tr(K)=0$.

If $\tr(A_0)=\tr(K)=0$, then also
$\tr(A_0+2K)=0$, so the correction terms in (iii) and (v) vanish.
Hence the matrices in (iii) and (v) are perfect squares for every
$t\neq0$.

Now consider the case $\tr(2A_0+K)\neq0$. Put
$$S_{\pm}=2A_0+K\pm U_t,\qquad E=\begin{pmatrix}0&1\\0&0\end{pmatrix}.$$
For the matrices appearing in (iii) and (v), respectively, we seek square
roots of the form
$$S_\pm\pm\gamma E,\qquad \gamma\in\mathbb Z.$$
Since $E^2=0$, we have
$$(S_\pm\pm\gamma E)^2=S_\pm^2\pm\gamma(S_\pm E+ES_\pm).$$
Therefore,
$$S_\pm^2\pm\frac{t}{2}\tr(A_0+2K)E=(S_\pm\pm\gamma E)^2$$
provided that
\begin{equation}\label{eq:t}
\frac{t}{2}\tr(A_0+2K)E=\gamma(S_\pm E+ES_\pm).
\end{equation}
Now,
$$S_\pm E+ES_\pm=\tr(S_\pm)E=\tr(2A_0+K\pm U_t)E=\tr(2A_0+K)E,$$
since $\tr U_t=0$. Hence, for
$$t=2\alpha\tr(2A_0+K),\qquad\alpha\in\mathbb Z\setminus\{0\},$$
and
$$\gamma=\alpha\tr(A_0+2K),$$
equation~\eqref{eq:t} holds. Consequently, the matrices in (iii) and
(v) are perfect squares.
\medskip

It remains to verify that the four matrices in \eqref{proto_quad} are distinct. Denote the four matrices in \eqref{proto_quad} by $B_1,\ldots,B_4$. 
\begin{itemize}
    \item $B_1=B_2$ $\iff$ $U_t=0$, but $U_t\not=0$ since $t\not=0$.
\item $B_1=B_3$ $\iff$ $U_t=A_0+2K$, but the identity $U_t=A_0+2K$ is possible for at most one value of $t$.\\
The same argument applies to the following three cases:
\item $B_1=B_4$ $\iff$ $U_t=4A_0+4K$,
\item $B_2=B_3$ $\iff$ $U_t=-A_0-2K$,
\item $B_2=B_4$ $\iff$ $U_t=-4A_0-4K$.
\item $B_3=B_4$ $\iff$ $3A_0+2K=0$, but $3A_0+2K\not =0$ by Lemma~\ref{lem:tech}. 
\end{itemize} 
Hence only finitely many values of $t$ must be excluded. Since the construction provides infinitely many admissible values of $t$, infinitely many remain for which the four matrices are pairwise distinct.

Moreover, after excluding finitely many values of $t$, we also have
$B_i\neq \mathbf 0$ for all $i=1,2,3,4$.
Indeed, by Lemma~\ref{lem:tech}, both
$2A_0+2K$ and $5A_0+4K$ are nonzero. The equalities
$A_0+U_t=0$ and $A_0-U_t=0$ can occur for at most one value of $t$ each, so these
values can be excluded together with the finitely many values producing coincidences.
\end{proof}

\begin{rem}
  In the construction of Theorem~\ref{thm:quadruple_for_diff_sq}, suppose in addition that
$$A_0^2\not=0, \quad \tr(A_0+K)\not=0,\quad  \tr(5A_0+4K)\not=0, \quad (A_0+K)\circ(5A_0+4K)\not=0.$$
Then infinitely many of the constructed Jordan $D(N)$-quadruples are \textbf{non-degenerate}, i.e. all their pairwise Jordan products are nonzero. Indeed, the two fixed products are
$$ (A_0+U_t)\circ(A_0-U_t)=A_0^2$$
and
$$(2A_0+2K)\circ(5A_0+4K)=2(A_0+K)\circ(5A_0+4K)$$
which are nonzero by assumption. The remaining four pairwise Jordan products $B_1\circ B_3$, $B_2\circ B_3$, $B_1\circ B_4$ and $B_2\circ B_4$ depend linearly on $t$. Their upper-right entries are affine linear functions of $t$, with coefficients $\tr(A_0+K)$ and $\frac{1}{2}\tr(5A_0+4K)$, respectively. Since these coefficients are nonzero by assumption, each vanishing condition excludes at most one value of $t$. Excluding finitely many such values, infinitely many admissible values of $t$ remain.

\end{rem}

Although not every representation $N=K^2-A_0^2$
is suitable for the construction of Theorem~\ref{thm:quadruple_for_diff_sq}, Lemma~\ref{lem:tech} guarantees the existence of a favourable representation. Hence every matrix $N\in UT_2(\Z)$ representable as a difference of two squares gives rise to infinitely many Jordan $D(N)$-quadruples. Combining this with the characterization of upper-triangular $2\times 2$ integer matrices representable as differences of two squares, given in \cite{diz2}, we obtain the following corollary.

\begin{cor}\label{cor:ex-of-quadruples}
Let
$$N=\begin{pmatrix}p&r\\0&q\end{pmatrix}\in UT_2(\Z).$$
If one of the following holds:
\begin{enumerate}
\item $p$ and $q$ are odd, and
\begin{itemize}
\item  $p\equiv q\pmod4$ and $r\equiv0\pmod2$;
\item $p\not\equiv q\pmod4$ and $r$ is arbitrary;
\end{itemize}
\item one of $p,q$ is odd and the other is divisible by $4$, and $r$ is arbitrary;
\item $p\equiv q\equiv0\pmod4$, and
\begin{itemize}
\item $(p,q)\mod16\in\{(4,4),(12,12)\}$ and $r\equiv0\pmod4$;
\item $(p,q)\mod16\in\{(4,12),(8,8),(12,4)\}$ and $r\equiv0\pmod2$;
\item otherwise, $r$ is arbitrary,
\end{itemize}
\end{enumerate}
then there exist infinitely many Jordan $D(N)$-quadruples in $UT_2(\Z)$.
\end{cor}

The remaining congruence classes therefore correspond precisely to matrices that are not representable as differences of two squares in $UT_2(\Z)$. 

\section{Matrices not representable as differences of two squares}

According to Corollary~\ref{cor:ex-of-quadruples}, it remains to investigate matrices that are not representable as differences of two squares in $UT_2(\Z)$, i.e. for the following matrix forms:
\begin{enumerate}
    \item[(I)] $p\equiv 2\pmod 4$ or $q\equiv 2\pmod 4$,
      \item[(II)] $p,q$ are odd, $p\equiv q\pmod4$ and $r\equiv1\pmod2$,
      \item[(III)] $(p,q)\equiv(4,4),(12,12)\pmod{16}$, $r\equiv1,2,3\pmod4$,
       \item[(IV)] $(p,q)\equiv(4,12),(8,8),(12,4)\pmod{16}$, $r\equiv1\pmod2$.\\
\end{enumerate}

Throughout this section, assume that
\begin{equation}\label{quadruple}
A_i=\begin{pmatrix}a_i&b_i\\0&d_i\end{pmatrix}, \qquad i=1,\ldots,4,
\end{equation}
have the Jordan $D(N)$-property in $UT_2(\Z)$, where $N=
\begin{pmatrix}
p&r\\0&q \end{pmatrix}$. Then, for every $1\le i<j\le 4$,
\begin{equation}\label{aiaj}
    A_i\circ A_j+N=\begin{pmatrix}
a_ia_j+p &
b_ic_j+b_jc_i+r\\0& d_id_j+q
\end{pmatrix}
\end{equation}
is a square in $UT_2(\Z)$, where we set
\begin{equation}\label{ci}
    c_k=\frac{a_k+d_k}{2}=\frac12\tr A_k,\qquad k=1,2,3,4. 
\end{equation}

The proofs of the \underline{non-existence results} rely on the following three ingredients:

\begin{enumerate}
\item[(i)] \textbf{Scalar  $D(n)$-conditions}\\
The diagonal entries of matrices $A_i$ satisfy
$$ a_ia_j + p = \square,\qquad d_id_j + q = \square,\qquad i \neq j,$$
i.e., they have the scalar $D(p)$- and $D(q)$-properties.
Consequently, known congruence restrictions for integer quadruples with the $D(n)$-property impose parity conditions on the diagonal entries of matrices $A_i$.

\item[(ii)] \textbf{Integrality condition} \\
Since $A_i \circ A_j = X_{ij}^2-N$, the Jordan product of two upper-triangular integer matrices $A_i$ and $A_j$ has integer entries. Hence, its upper-right entry satisfies additional integrality (and consequently parity) constraints. 

\item[(iii)] \textbf{Square condition}\\
Finally, if $A_i\circ A_j+N$ is a square in $UT_2(\Z)$, formula~\eqref{aiaj} imposes conditions on the upper-right entry. {\color{black} By \eqref{kvadrat}, if the trace of a square root is nonzero, then the upper-right entry in \eqref{aiaj} is divisible by this trace; if the trace is zero, then this entry must be zero.}
\end{enumerate}

The interplay between these three ingredients yields the desired contradictions. In addition, the following technical lemma is often used.
\begin{lem}[$\Z_2^2$-argument]\label{lem:tech2}
    There do not exist integers $x_i$, $y_i$, $i=1,2,3,4$, such that 
    $$ x_iy_j+x_jy_i\equiv 1\pmod 2, \qquad 1\le i<j\le 4.$$
\end{lem}
\begin{proof}
    Let $v_i=(x_i,y_i)\mod 2\in\Z_2^2$, $i=1,2,3,4$. Then
$$x_i y_j+x_j y_i\equiv\det\begin{pmatrix}v_i\\v_j\end{pmatrix}\pmod2,$$
so the assumptions are equivalent to
$$\det\begin{pmatrix}v_i\\v_j\end{pmatrix}=1,\qquad1\le i<j\le4.$$
Over $\Z_2$, two vectors have determinant $1$ if and only if they are distinct and nonzero. Since $\Z_2^2$ contains only three nonzero vectors, four vectors satisfying the above condition cannot exist.
\end{proof}

\subsection{Case I} $p\equiv 2\pmod 4$ or $q\equiv 2\pmod 4$.

\begin{prop}\label{prop:diag-2mod4}
Let
$$N=\begin{pmatrix}p&r\\0&q\end{pmatrix}\in UT_2(\mathbb Z).$$
If $p\equiv 2\pmod 4$ or $q\equiv 2\pmod 4$, then there are no four matrices in $UT_2(\mathbb Z)$ having the Jordan $D(N)$-property.
Consequently, there is no upper-triangular Jordan $D(N)$-quadruple.
\end{prop}

\begin{proof} {\color{black}
Assume first that $p\equiv2\pmod4$ and that the four matrices in \eqref{quadruple} have the Jordan $D(N)$-property. We use the modulo $4$ argument
from the proof of Theorem~1 in \cite{bro}. We include it here to note
that the argument remains valid when repetitions and zero entries are
allowed.

If two of the integers $a_1,\ldots,a_4$ were even, say $a_i$ and
$a_j$, then
\[
a_i a_j+p\equiv2\pmod4,
\]
which is impossible for a square. Hence at most one of the $a_i$ is
even, so at least three of them are odd.

For any two odd entries $a_i,a_j$, the condition
$$a_i a_j+p=\square$$
implies
$$a_i a_j\equiv3\pmod4.$$
Thus any two of the odd entries must have different residues modulo
$4$. This is impossible for three odd integers, since there are only
two odd residue classes modulo $4$.

The case $q\equiv2\pmod4$ is analogous.}
\end{proof}

\subsection{Case II} $p,q$ are odd, $p\equiv q\pmod4$ and $r\equiv1\pmod2$.

\begin{prop}\label{prop:4k3-4n3-odd}
Let
$$N=\begin{pmatrix}4k+3&2m+1\\0&4n+3\end{pmatrix}\in UT_2(\Z).$$
Then there are no four matrices in $UT_2(\mathbb Z)$ having the Jordan $D(N)$-property.
Consequently, there is no upper-triangular Jordan $D(N)$-quadruple.
\end{prop}

\begin{proof}
Assume that the four matrices in \eqref{quadruple} have the Jordan $D(N)$-property. By the scalar $D(n)$-conditions, for every $i\ne j$, the diagonal entries of
$A_i\circ A_j+N$ satisfy
$$a_i a_j+(4k+3)=\square,\qquad d_i d_j+(4n+3)=\square.$$
Hence, among  $a_1,a_2,a_3,a_4$ there is at most one even element,
and the same holds for  $d_1,d_2,d_3,d_4$. Therefore the following parity patterns are possible:
\begin{itemize}
    \item[(a)] all $a_i$ and $d_i$ are odd,
\item[(b)] exactly one $a_i$ is even, all $d_i$ are odd,
\item[(c)] exactly one $d_i$ is even, all $a_i$ are odd,
\item[(d)] exactly one $a_i$ is even and exactly one $d_i$ is even.
\end{itemize}

\textbf{Case (a):} Since  $a_i+d_i$ is even, $c_i\in\Z$ for $i=1,2,3,4$. 
Therefore, the upper-right entry of $A_i\circ A_j+N$ is an integer: 
$$ [A_i\circ A_j+N]_{12}=b_ic_j+b_jc_i+2m+1,$$
for all $1\le i<j\le 4$. Since the diagonal entries of each $A_i\circ A_j+N$ are even, the square condition implies that its upper-right entry is even. Hence
$$ b_ic_j+b_jc_i\equiv 1\pmod 2, ~ 1\le i<j\le 4,$$
which is not possible by Lemma \ref{lem:tech2}.

\textbf{Case (b):} Assume without loss of generality that $a_1$ is even. Then 
$$a_1+d_1\equiv1\pmod 2,~ a_i+d_i\equiv 0\pmod 2, ~  c_i\in\Z, ~i=2,3,4.  $$
By the integrality condition applied to $A_1\circ A_j$, $j=2,3,4$, we get
$$ b_j(a_1+d_1)+b_1(a_j+d_j)\equiv 0\pmod2.$$
Therefore $b_j$ is even for $j=2,3,4$. Hence
$$[A_i\circ A_j+N]_{12}=b_i c_j+b_jc_i+2m+1\equiv 1\pmod2, ~ 2\le i<j\le 4 . $$
On the other hand, for $i,j\in\{2,3,4\}$, all $a_i,d_i$ are odd. Thus, diagonal entries of a square root of $A_i\circ A_j+N$ are even. Hence  the square condition gives
$$[A_i\circ A_j+N]_{12}\equiv 0\pmod2, ~ 2\le i<j\le 4 , $$
a contradiction.

\textbf{Case (c):} Analogous to Case (b).

\textbf{Case (d):} There are two subcases: if even diagonal elements occur in the same matrix or if they occur in different matrices. 

\textbf{(d1)} If the even entries occur in the same matrix, say $a_1$ and $d_1$, then $a_i+d_i$ is even for all $i$, so the integrality condition is satisfied. Note that diagonal entries of each matrix $A_i\circ A_j+N$ are of the same parity. Indeed, 
$$ [A_1\circ A_j+N]_{11}=a_1a_j+4k+3\equiv1\pmod2, \quad j=2,3,4,$$ 
$$  [A_1\circ A_j+N]_{22}=d_1d_j+4n+3\equiv1\pmod2,\quad j=2,3,4,$$
and 
$$ [A_i\circ A_j+N]_{kk}\equiv0\pmod2,\quad k=1,2,\quad 2\le i<j\le4. $$
So, the trace of a square root of each $A_i\circ A_j+N$ is divisible by 2. Hence the square condition imposes that 
 $$ b_ic_j+b_jc_i+1\equiv 0\pmod 2, \quad 1\le i<j\le 4,$$
which is not possible by Lemma \ref{lem:tech2}.

\textbf{(d2)} If the even entries occur in different matrices, say $a_1$ and $d_2$ are even,
then for $j=3,4$, the integrality condition applied to $A_1\circ A_j$ and $A_2\circ A_j$ forces $b_3$ and $b_4$ to be even. Indeed, since $a_1$ and $d_1$ are of opposite parities and $c_3,c_4\in\Z$ because $a_3, d_3, a_4, d_4$ are odd, we have
$$ \frac{a_1+d_1}{2}b_j+c_jb_1\in\Z~\implies~ b_j\equiv0\pmod2,~ j=3,4. $$
Hence 
$$[A_3\circ A_4+N]_{12}=b_3c_4+b_4c_3+2m+1\equiv1\pmod2.$$ 
But $a_3,a_4,d_3,d_4$ are all odd, i.e. diagonal entries of $A_3\circ A_4+N$ are even,  so the square condition forces $[A_3\circ A_4+N]_{12}$ to be even, a contradiction.
\end{proof}

The previous proposition shows that Jordan $D(N)$-quadruples do not exist for one of the congruence classes arising in Case II ($p\equiv q \equiv 3 \pmod4$). We now show that the complementary congruence class ($p\equiv q \equiv 1 \pmod4$) behaves differently by constructing explicit quadruples.

\begin{prop}\label{prop:quadruple-1m1}
If 
$$N=\begin{pmatrix}1&m\\0&1\end{pmatrix}, ~ m\in\Z,$$
then the matrices
$$A_1=\begin{pmatrix}1&-24m\\0&8\end{pmatrix},
A_2=\begin{pmatrix}3&12m\\0&120\end{pmatrix},A_3=\begin{pmatrix}8&-2m\\0&3\end{pmatrix},
A_4=\begin{pmatrix}120&6m\\0&1\end{pmatrix}$$
form an upper-triangular Jordan $D(N)$-quadruple consisting of nonsingular matrices. 
\end{prop}
\begin{rem}
    Note that the top and bottom diagonal entries are the two permutations $(1,3,8,120)$ and $(8,120,3,1)$ of a scalar $D(1)$-quadruple, namely Fermat's quadruple $\{1,3,8,120\}$.
\end{rem}

\begin{proof}
A routine computation yields
$$A_1\circ A_2+N=\begin{pmatrix}4&-1421m\\0&961\end{pmatrix}=
\begin{pmatrix}2&49m\\0&-31\end{pmatrix}^2,$$
$$A_1\circ A_3+N=\begin{pmatrix}9&-140m\\0&25\end{pmatrix}=
\begin{pmatrix}-3&-70m\\0&5\end{pmatrix}^2,$$
$$A_1\circ A_4+N=\begin{pmatrix}121&-1424m\\0&9\end{pmatrix}=
\begin{pmatrix}11&-178m\\0&-3\end{pmatrix}^2,$$
$$A_2\circ A_3+N=\begin{pmatrix}25&-56m\\0&361\end{pmatrix}=
\begin{pmatrix}-5&-4m\\0&19\end{pmatrix}^2,$$
$$A_2\circ A_4+N=\begin{pmatrix}361&1096m\\0&121\end{pmatrix}=
\begin{pmatrix}19&137m\\0&-11\end{pmatrix}^2,$$
$$A_3\circ A_4+N=\begin{pmatrix}961&-87m\\0&4\end{pmatrix}=
\begin{pmatrix}-31&3m\\0&2\end{pmatrix}^2.$$
\end{proof}

\begin{prop}\label{prop:quintuple}
For $N=\begin{pmatrix}1&1\\0&1\end{pmatrix},$ 
the five matrices
\[
A=\begin{pmatrix}0&0\\0&1\end{pmatrix},\ \
B=\begin{pmatrix}0&0\\0&3\end{pmatrix},\ \
C=\begin{pmatrix}0&2\\0&0\end{pmatrix},\ \
D=\begin{pmatrix}1&2\\0&0\end{pmatrix},\ \
E=\begin{pmatrix}3&2\\0&0\end{pmatrix}
\]
form a Jordan $D(N)$-quintuple. Moreover, this quintuple cannot be extended to a Jordan $D(N)$-sixtuple by any upper-triangular integer matrix.
\end{prop}

\begin{proof}
Directly,
$$A\circ B+N=\begin{pmatrix}1&1\\0&4\end{pmatrix}=
\begin{pmatrix}-1&1\\0&2\end{pmatrix}^2,$$
$$A\circ C+N=A\circ D+N=A\circ E+N=C\circ D+N=\begin{pmatrix}1&2\\0&1\end{pmatrix}=
\begin{pmatrix}1&1\\0&1\end{pmatrix}^2,$$
$$B\circ C+N=B\circ D+N=B\circ E+N=C\circ E+N=\begin{pmatrix}1&4\\0&1\end{pmatrix}=
\begin{pmatrix}1&2\\0&1\end{pmatrix}^2,$$
$$D\circ E+N=\begin{pmatrix}4&5\\0&1\end{pmatrix}=
\begin{pmatrix}2&5\\0&-1\end{pmatrix}^2.$$
So the five matrices form a Jordan $D(N)$-quintuple.

Now suppose
$F=\begin{pmatrix}a&b\\0&d\end{pmatrix}$
extends it to a Jordan $D(N)$-sixtuple. From the pairs $(F,D)$ and $(F,E)$ one gets that $a+1$ and $3a+1$ must both be squares, hence $a$ is even. Similarly, from $(F,A)$ and $(F,B)$ one gets that $d+1$ and $3d+1$ are squares, hence $d$ is even. But then
\[
C\circ F+N=\begin{pmatrix}1&a+d+1\\0&1\end{pmatrix}
\]
has odd upper-right entry, while by  the form  \eqref{kvadrat},  a square with diagonal entries $1,1$ must have even upper-right entry. Contradiction.
\end{proof}

\begin{rem} 
Note that, by \cite[Theorem~11]{diz2}, $N$ from Proposition \ref{prop:quadruple-1m1} can be represented as a difference of two squares in $UT_2(\Z)$ if and only if $m$ is even. Thus, for odd $m$, Proposition \ref{prop:quadruple-1m1} provides Jordan $D(N)$-quadruples for matrices that are not representable as differences of two squares. This indicates that the case $p\equiv q\equiv1\pmod4$ will be quite different.
\end{rem}

\begin{prop}\label{prop:8k1-8n5-odd}
Let $$N=\begin{pmatrix}8k+1&2m+1\\0&8n+5\end{pmatrix}~ \text{ or }~ N=\begin{pmatrix}8k+5&2m+1\\0&8n+1\end{pmatrix}\in UT_2(\Z).$$ 
Then there are no four matrices in $UT_2(\mathbb Z)$ having the Jordan $D(N)$-property.
Consequently, there is no upper-triangular Jordan $D(N)$-quadruple.
\end{prop}
\begin{proof}
It suffices to consider
$$N=\begin{pmatrix}8k+1&2m+1\\0&8n+5\end{pmatrix},$$
since the other case is symmetric. Assume that the four matrices in \eqref{quadruple} have the Jordan $D(N)$-property. By the scalar $D(n)$-conditions, the top diagonal entries satisfy
 $$a_i a_j+1\equiv \square \pmod 8,\qquad i\ne j.$$  
 Thus, up to permutation, the possible residue multisets $\{a_1,a_2,a_3,a_4\}$ modulo 8 are:
$$(0,0,0,r),r\in\Z_8,$$$$
(0,0,1,3), (0,0,1,7), (0,0,2,4), (0,0,3,5),
(0,0,4,4), (0,0,4,6),   (0,0,5,7),$$$$
(0,2,4,4), (0,4,4,4), (0,4,4,6),
(2,4,4,4), (4,4,4,4), (4,4,4,6).$$
Similarly, the bottom diagonal entries satisfy
$$d_i d_j+5\equiv \square \pmod 8,\qquad i\ne j,$$  
and hence, up to permutation,
the possible residue multisets $\{d_1,d_2,d_3,d_4\}$ modulo 8 are more restrictive:
$$(2,2,2,2), (2,2,2,6), (2,2,6,6), (2,6,6,6), (6,6,6,6).$$
In particular, all $d_i\equiv2\pmod4$.

We consider, up to permutation, the following cases:
\begin{enumerate}
\item[(a)] all $a_i$ are even,
\item[(b)] $(a_1,a_2,a_3,a_4)\equiv(0,0,0,\rho)\pmod8$, where $\rho$ is odd,
\item[(c)] $(a_1,a_2,a_3,a_4)\equiv(0,0,1,3),(0,0,1,7),(0,0,3,5),(0,0,5,7)\pmod8$. 
\end{enumerate}

\textbf{Case (a):} Since all $a_i$ and $d_i$
are even, all $a_i+d_i$ are even and $c_i\in\Z$, $i=1,2,3,4$. Moreover, for all $i\neq j$ both diagonal entries of $A_i\circ A_j+N$ are odd, so the square condition implies that the upper-right entry of $A_i\circ A_j+N$ is even. Since $2m+1$ is odd, \eqref{aiaj} gives
$$b_i c_j+b_j c_i\equiv1\pmod2,\qquad 1\le i<j\le4.$$
This is impossible by Lemma \ref{lem:tech2}.

\textbf{Case (b):} Assume
$$a_1\equiv a_2\equiv a_3\equiv0\pmod8, \qquad a_4\equiv1\pmod2.$$
Since $d_i\equiv2\pmod4$, $i=1,2,3,4$, we have
$$a_i+d_i\equiv 2\pmod4,~ i=1,2,3,\qquad a_4+d_4\equiv 1\pmod2.$$
By the integrality condition applied to $A_i\circ A_4$, $i=1,2,3$, it follows that $b_1,b_2,b_3$ are even. Now consider $A_1\circ A_2+N$. Its diagonal entries are odd, so the square condition forces its upper-right entry to be even. On the other hand, using \eqref{aiaj} and the fact that $b_1,b_2$ are even, we obtain
$$[A_1\circ A_2+N]_{12}=b_1c_2+b_2c_1+2m+1\equiv1\pmod2,$$
a contradiction.

\textbf{Case (c):} We may assume
$$a_1\equiv a_2\equiv0\pmod8,\qquad a_3\equiv a_4\equiv 1\pmod 2.$$
Again,
$$
a_1+d_1\equiv a_2+d_2\equiv2\pmod4,$$
while
$$a_3+d_3\equiv a_4+d_4\equiv1\pmod2.$$
By the integrality condition applied to $A_1\circ A_3$ and $A_1\circ A_4$, we get
$b_1\equiv 0\pmod2.$
Similarly, integrality of $A_2\circ A_3$ and $A_2\circ A_4$ forces $b_2\equiv 0\pmod2.$
Hence, 
$$[A_1\circ A_2+N]_{12}= b_1c_2+b_2c_1+2m+1\equiv1\pmod2.$$
But the diagonal entries of $A_1\circ A_2+N$ are odd, so the square condition requires its upper-right entry to be even, a contradiction.

Thus all possible residue patterns lead to contradictions, and hence no Jordan $D(N)$-quadruple exists.
\end{proof}
 
\begin{prop}\label{prop:8k5-8n5-odd}
Let $$N=\begin{pmatrix}8k+5&2m+1\\0&8n+5\end{pmatrix}\in UT_2(\Z).$$ 
Then there are no four matrices in $UT_2(\mathbb Z)$ having the Jordan $D(N)$-property.
Consequently, there is no upper-triangular Jordan $D(N)$-quadruple.
\end{prop}
\begin{proof}
Assume that the four matrices in \eqref{quadruple} have the Jordan $D(N)$-property.  By the scalar $D(n)$-conditions,
$$a_ia_j+8k+5=\square,~ d_id_j+8n+5=\square,~ 1\le i<j\le 4.$$
Reducing modulo 8, this implies
$$a_i,d_i\equiv2 \text{ or }6\pmod8,~ i=1,2,3,4.$$
Hence all $a_i$ and $d_i$ are even. Since $[N]_{12}=2m+1$ is odd, we are exactly in Case~(a) of the proof of Proposition~\ref{prop:8k1-8n5-odd}, which gives a contradiction.
\end{proof}

{\color{black}
\begin{prop}\label{prop:s2-m-t2}
Let  $s, t \ge 0$, $r \in \Z$, and $$N_r=\begin{pmatrix}
   s^2& r\\ 0& t^2
\end{pmatrix}.$$
Then $N_r$ admits infinitely many Jordan $D(N_r)$-quadruples in $UT_2(\Z)$. 
\end{prop}
\begin{proof}
We first consider the case $r=0$. 
Since $N_0$ is representable as a difference of two squares in $UT_2(\Z)$,
 by Theorem~\ref{thm:quadruple_for_diff_sq}, $N_0$ admits infinitely many Jordan
$D(N_0)$-quadruples.

Now assume that $r\neq0$. It suffices first to consider $r=1$. Indeed, once a Jordan
$D(N_1)$-quadruple is constructed for
$$N_1=\begin{pmatrix}s^2&1\\0&t^2\end{pmatrix},$$
Lemma~\ref{lem:scale-m} yields a Jordan $D(N_r)$-quadruple for every
$r\neq0$ by multiplying all upper-right entries by $r$.
Let
$$A_1=\begin{pmatrix}0&-1\\0&1\end{pmatrix},\qquad
A_2=\begin{pmatrix}1&-1\\0&0\end{pmatrix},$$
$$A_3=\begin{pmatrix}2s+1&b_3\\0&4t+4\end{pmatrix},\qquad
A_4=\begin{pmatrix}4s+4&b_4\\0&2t+1\end{pmatrix},$$
where $b_3,b_4\in\Z$  will be chosen below.
We are looking for $b_3,b_4\in\Z$ such that the four matrices form a Jordan $D(N_1)$-quadruple. A direct computation gives
$$A_1\circ A_2+N_1=
\begin{pmatrix}s&0\\0&t\end{pmatrix}^2,~ A_1\circ A_3+N_1=
\begin{pmatrix}s^2 & \frac{1}{2} (b_3-2 s-4 t-3) \\
 0 & (t+2)^2\end{pmatrix},$$
$$A_1\circ A_4+N_1=
\begin{pmatrix}s^2 & \frac{1}{2} (b_4-4 s-2 t-3) \\
 0 & (t+1)^2\end{pmatrix},$$
 $$ A_2\circ A_3+N_1=
\begin{pmatrix}(s+1)^2 & \frac{1}{2} (b_3-2 s-4 t-3) \\
 0 & t^2\end{pmatrix},$$
$$A_2\circ A_4+N_1=
\begin{pmatrix}(s+2)^2 & \frac{1}{2} (b_4-4 s-2 t-3) \\
 0 & t^2\end{pmatrix},$$
 $$
A_3\circ A_4+N_1=
\begin{pmatrix}(3 s+2)^2 & \frac{1}{2} (2 + b_3 (5 + 4 s + 2 t) + b_4 (5 + 2 s + 4 t)) \\
 0 & (3 t+2)^2\end{pmatrix}.$$
To make $A_1\circ A_3+N_1$ and $A_2\circ A_3+N_1$  squares, we choose $b_3$ so that their common upper-right entry is divisible by both corresponding traces. Set
$$b_3= (2 + s + t) (2 + 2X (1 + s + t))-1 + 2 t,\qquad X\in\Z .  $$
Then
$$[A_1\circ A_3+N_1]_{12}=[A_2\circ A_3+N_1]_{12}=(1 + s + t) (2 + s + t) X .$$
Since the traces of the corresponding diagonal square roots are $2 + s + t$ and
$1 + s + t$, respectively, both matrices are squares in $UT_2(\Z)$.

Similarly, choosing 
$$b_4= (1 + s + t) (4 - 2(2 + s + t) Y)-1 - 2 t,\qquad Y\in\Z  $$
gives 
$$[A_1\circ A_4+N_1]_{12}=[A_2\circ A_4+N_1]_{12}=-  (1 + s + t) (2 + s + t) Y ,$$
so $A_1\circ A_4+N_1$ and $A_2\circ A_4+N_1$ are squares as well.

It remains to consider $A_3\circ A_4+N_1$. Set
$$ S=1+s+t,~ P=2+2s+t,~ Q=2+s+2t.$$ 
After substituting the above expressions for $b_3$ and $b_4$, the upper-right entry becomes
$$ [A_3\circ A_4+N_1]_{12}=S(S + 1)(2P + 1)X - S(S + 1)(2Q + 1)Y+ 4PQ.$$
The trace of the indicated diagonal square root is
$$ (3s + 2) + (3t + 2) = 3S + 1.$$
Thus it suffices to solve
\begin{equation}\label{eq:cong}
  S(S+1)(2P+1)X-S(S+1)(2Q+1)Y\equiv-4PQ\pmod{3S+1}.
\end{equation}

The congruence is solvable if
$$g:=\gcd\bigl(S(S+1)(2P+1),S(S+1)(2Q+1),3S+1\bigr)\mid 4PQ.$$
We show that $g\mid 2$. First note that
$$\gcd(S,3S+1)=1,\qquad \gcd(S+1,3S+1)\mid 2,$$
and
$$(2P+1)+(2Q+1)-2(3S+1)=2.$$
Let $p$ be an odd prime divisor of $g$. Since $p\mid 3S+1$,
the first two relations imply that $p\nmid S(S+1)$. Hence, from
$$p\mid S(S+1)(2P+1),\qquad p\mid S(S+1)(2Q+1),$$
we obtain
$$p\mid 2P+1,\qquad p\mid 2Q+1.$$
The last displayed relation then gives $p\mid2$, a contradiction.
Thus $g$ has no odd prime divisors, and therefore $g$ is a power of $2$.

Since $2P+1$ and $2Q+1$ are odd, it follows that
$$g\mid S(S+1).$$
If $S$ is even, then $3S+1$ is odd, and hence $g=1$.
If $S$ is odd, then $g\mid S+1$. Since also $g\mid3S+1$, we obtain
\[
g\mid 3(S+1)-(3S+1)=2.
\]
Thus, in all cases, $g\mid2$. Consequently, $g\mid4PQ$, and
\eqref{eq:cong} is solvable.

Moreover, once one solution $(X_0, Y_0)$ is found, $(X_0 + (3S + 1)k, Y_0)$ is again
a solution for every $k \in Z$. Since $S(S+1)\neq0$, these choices give infinitely many distinct values of $b_3$, and hence infinitely many distinct Jordan $D(N_1)$-
quadruples. Note that the diagonal pairs of $A_1, A_2, A_3, A_4$ are pairwise distinct for $s, t \ge 0$. Hence all four matrices are distinct
and nonzero for every $s,t,X,Y$.
\end{proof}

\begin{rem}
Proposition \ref{prop:s2-m-t2} shows that the remaining part of Case II contains a large positive subfamily: whenever both diagonal entries are squares, a Jordan $D(N)$-quadruple exists. In particular, this applies to squares congruent to $1$ modulo $8$ with odd upper-right entry.
\end{rem}

}

\subsection{Case III} $(p,q)\equiv(4,4),(12,12)\pmod{16}$, $r\equiv1,2,3\pmod4$.

\begin{prop}\label{prop:16k12-16n12-m}
Let $m\in\Z$, $4\nmid m$ and  $$ N=
\begin{pmatrix}
16k+12 & m\\
0 & 16n+12
\end{pmatrix}\in UT_2(\Z).$$ Then there is no  Jordan $D(N)$-quadruple in $ UT_2(\Z)$.
\end{prop}
\begin{proof}
Assume that the four matrices in \eqref{quadruple} have the Jordan $D(N)$-property.  By the scalar $D(n)$-conditions, for $i\not= j$
$$a_i a_j+16k+12=\square,
\qquad
d_i d_j+16n+12=\square.$$
Since the quadratic residues modulo $16$ are $0,1,4,9$, a direct check modulo $16$ shows that, up to permutation,
$$(a_1,a_2,a_3,a_4)\equiv (2,2,2,2)
\quad\text{or}\quad
(0,2,2,2)\pmod4.$$
The same conclusion holds for $(d_1,d_2,d_3,d_4)$. In particular, $a_i+d_i$
 is even for every $i$, so $c_i=(a_i+d_i)/2\in\Z$.

\textbf{Case 1}: $m$ is odd

Note that we are in the same situation as Case (a) of Proposition \ref{prop:4k3-4n3-odd}. Indeed, all diagonal entries of $A_i\circ A_j+N$ are even squares, so its upper-right entry is even. Since $m$ is odd, this forces $c_ib_j+c_jb_i$ 
 to be odd for every $i<j$, and the same $\Z_2^2$-argument gives a contradiction.

\textbf{Case 2}: $m\equiv 2\pmod4$

Up to permutations of the indices and symmetry between the top and bottom diagonal entries, it suffices to consider the following four subcases.

\textbf{Subcases}:
$$\textbf{2a)}\quad \begin{matrix}
    a_i\mod 4 &:& 2&2&2 &2 \\
    d_i\mod 4 &:& 2&2&2 &2\\
\end{matrix}\,\qquad \textbf{2b)}\quad \begin{matrix}
    a_i\mod 4 &:& 2&2&2 &2 \\
    d_i\mod 4 &:& 2&2&2 &0
\end{matrix}\ $$ 
$$\textbf{2c)}\quad \begin{matrix}
    a_i\mod 4 &:& 2&2&2 &0 \\
    d_i\mod 4 &:& 2&2&2 &0
\end{matrix}\,\qquad \textbf{2d)}\quad \begin{matrix}
    a_i\mod 4 &:& 2&2&2 &0 \\
    d_i\mod 4 &:& 2&2&0 &2
\end{matrix}\ .$$ 

We analyze Subcases 2a) and 2c) simultaneously, and then Subcases 2b) and 2d).

\textbf{Subcases 2a and 2c}: In both subcases,
$$a_i+d_i\equiv0\pmod4,$$
and hence
$$ c_i\equiv0\pmod{2},\qquad i=1,2,3,4.$$
We claim that, for every pair $i<j$, every upper-triangular square root of $A_i\circ A_j+N$ has trace divisible by $4$.
 In Subcase 2a), for every $i<j$,
$$[A_i\circ A_j+N]_{11}=a_ia_j+16k+12\equiv 4+0+4\equiv0\pmod 8,$$
and similarly,
$$[A_i\circ A_j+N]_{22}=d_id_j+16n+12\equiv 0\pmod 8. $$
Since these entries are squares, they are in fact congruent to $0\pmod{16}$. Thus both corresponding diagonal entries of a square root are divisible by $4$.

In Subcase 2c), for $1\le i<j\le3$,
$$[A_i\circ A_j+N]_{11}\equiv [A_i\circ A_j+N]_{22}\equiv0\pmod8, $$
so both diagonal entries are congruent to $0\pmod{16}$. On the other hand, for $i=1,2,3$,
$$[A_i\circ A_4+N]_{11}\equiv [A_i\circ A_4+N]_{22}\equiv4\pmod8. $$
So, both are congruent to $4\pmod{16}$, and the corresponding diagonal entries of a square root are both congruent to $2\pmod4$.

Therefore, in every pair, the sum of the diagonal entries of a square root is divisible by $4$. Hence the upper-right entry of the square is divisible by $4$, and
$$ 4\mid [A_i\circ A_j+N]_{12}=c_ib_j+c_jb_i+4\ell+2, $$
where $m=4\ell+2$. Finally, we obtain 
$$c_ib_j+c_jb_i\equiv2\pmod{4},\qquad  1\le i<j\le4.$$
By putting $c_i=2e_i$, for all $i$, we get
$$e_ib_j+e_jb_i\equiv1\pmod{2},\qquad  1\le i<j\le4.$$
This is impossible by Lemma \ref{lem:tech2}.

\textbf{Subcases 2b and 2d}: In both subcases, we have
$$ c_1\equiv c_2\equiv0\pmod2,\qquad c_4\equiv1\pmod2. $$
Notice that the parity of $c_3$ is not needed in the argument.
Since the diagonal elements of $A_i\circ A_j+N$ are even, the square condition implies that 
$$ c_ib_j+c_jb_i\equiv0\pmod{2},\qquad  1\le i<j\le4.$$
In particular, for $(i,j)=(1,4)$ and $(2,4)$, since $c_1$ and $c_2$ are even, and $c_4$ is odd,  we get 
$$ c_1b_4+c_4b_1\equiv0\pmod{2}~ \implies b_1\equiv0\pmod2,$$
$$ c_2b_4+c_4b_2\equiv0\pmod{2}~ \implies b_2\equiv0\pmod2.$$
Hence,
\begin{equation*}\label{cibi}
    c_1b_2+c_2b_1\equiv0\pmod{4}.
\end{equation*}
For the pair $(1,2)$, both diagonal entries are congruent to $0\pmod{16}$, so the diagonal entries of a square root are divisible by $4$ and therefore
$$ c_1b_2+c_2b_1\equiv2\pmod{4},$$
contradicting the previous congruence.
\end{proof}

\begin{rem}
Proposition~\ref{prop:s2-m-t2} settles positively the
square-diagonal subfamily of the remaining class
$(p,q)\equiv(4,4)\pmod{16}$ with $4\nmid m$.
\end{rem}

In the next proposition we give a family of Jordan $D(N_t)$-quadruples for which $N_t$ is representable as a difference of two squares when $t$ is even, but is not representable when $t$ is odd. It follows immediately from Example \ref{exa.duje} and Lemma \ref{lem:scale-m} by scaling all
upper-right entries by $t$.

\begin{prop}\label{prop:4-2t-4}
For every integer $t\not =0$, the matrices
\[N_t=\begin{pmatrix}4&2t\\0&4\end{pmatrix},\]
\[A_t=\begin{pmatrix}0&t\\0&0\end{pmatrix},\quad
B_t=\begin{pmatrix}0&2t\\0&12\end{pmatrix},\quad
C_t=\begin{pmatrix}3&2t\\0&1\end{pmatrix},\quad
D_t=\begin{pmatrix}7&t\\0&5\end{pmatrix}\]
form an upper-triangular Jordan $D(N_t)$-quadruple.
\end{prop}

\subsection{Case IV} $(p,q)\equiv(4,12),(8,8),(12,4)\pmod{16}$, $r\equiv1\pmod2$.

\begin{prop}\label{prop:16k4-odd-16m12} Let
$$N=\begin{pmatrix}16k+4&2m+1\\0&16n+12\end{pmatrix} \text{ or } N=\begin{pmatrix}16k+12&2m+1\\0&16n+4\end{pmatrix} \in UT_2(\Z). $$  
Then there are no four matrices in $UT_2(\mathbb Z)$ having the Jordan $D(N)$-property.
Consequently, there is no upper-triangular Jordan $D(N)$-quadruple.
\end{prop}
\begin{proof}
   Assume that the four matrices in \eqref{quadruple} have the Jordan $D(N)$-property. The scalar $D(n)$-conditions give 
$$(a_1, a_2, a_3, a_4) \equiv (0, 0, 0, 0) \text{ or } (0, 0, 0, \rho) \text{ or } (0, 0, \rho, \rho)\pmod 4, \rho\in\{1,2,3\},$$
$$(d_1, d_2, d_3, d_4) \equiv (2, 2, 2, 2) \text{ or } (0, 2, 2, 2) \pmod 4.$$

We have several cases to discuss.
\begin{itemize}
    \item If all $a_i$ are even, then all $a_i+d_i$ are even, and the same $\Z_2^2$--argument as in Case (a) of Proposition \ref{prop:4k3-4n3-odd} yields a contradiction.
    
 \item Assume that $a_i\equiv d_i+2\equiv 0\pmod{4}$ for $i=1,2,3$, i.e.
$$\begin{matrix}
    a_i\mod 4 &:& 0&0&0 &* \\
    d_i\mod 4 &:& 2&2&2 &*\\
\end{matrix}$$
where $*$ denotes an entry that is not relevant to the argument. Define
$$ c_i=\frac12(a_i+d_i)\equiv1\pmod2, ~ i=1,2,3. $$
Then for $1\le i<j\le 3$ we have
$$ [A_i\circ A_j+N]_{12}=b_i c_j+b_j c_i+2m+1\equiv b_i+b_j+1. $$
By the square condition (since diagonal elements are even), we get
$$ b_1+b_2\equiv 1\pmod{2},~~ b_1+b_3\equiv 1\pmod{2},~~ b_2+b_3\equiv 1\pmod{2},$$
which is not possible.   
\item Assume $\rho$ is odd and 
$$\begin{matrix}
    a_i\mod 4 &:& 0&0&0 &\rho \\
    d_i\mod 4 &:& 0&2&2 &2\\
\end{matrix}.$$
Hence, with $c_i$ defined as above: 
$$ c_1\equiv 0\pmod2,\qquad c_2\equiv c_3\equiv 1\pmod2.$$
The  square condition (since diagonal elements are even) for the pair of indices (1,2) implies:
$$ b_1\equiv 1\pmod2.  $$
On the other hand the integrality condition for indices (1,4)  forces
$$ b_1(a_4+d_4)+b_4(a_1+d_1)\equiv b_1\equiv 0\pmod2,$$
a contradiction.
\item Assume $\rho$ is odd and 
$$\begin{matrix}
    a_i\mod 4 &:& 0&0&\rho &\rho \\
    d_i\mod 4 &:& 2&2&2 &*\\
\end{matrix}.$$
Hence, the $c_i$ defined as in the previous case satisfy 
$$ c_1\equiv c_2\equiv 1\pmod2.$$
The  square condition for the pair of indices (1,2) (since diagonal elements are even)  implies:
$$ b_1+b_2\equiv 1\pmod{2}.$$
The integrality condition for the pairs $(1,3)$ and $(2,3)$ gives
$$ b_1(a_3+d_3)+b_3(2c_1)\equiv b_1\equiv 0\pmod2,$$
$$ b_2(a_3+d_3)+b_3(2c_2)\equiv b_2\equiv 0\pmod2.$$
Therefore, $b_1+b_2\equiv 0\pmod2$, a contradiction.
\item Assume $\rho$ is odd and 
$$\begin{matrix}
    a_i\mod 4 &:& 0&0&\rho &\rho \\
    d_i\mod 4 &:& 0&2&2 &2\\
\end{matrix}.$$
We have 
$$ c_1\equiv 0\pmod2,\qquad c_2\equiv1\pmod2.$$
The integrality condition for indices (1,3) gives
$$ b_1(a_3+d_3)+b_3(a_1+d_1)\equiv b_1\equiv 0\pmod2.$$
The  square condition for  indices (1,2) implies:
$$ b_1c_2+b_2c_1+2m+1\equiv 0\pmod2,$$
but the LHS is odd, a contradiction.
\end{itemize}

\end{proof}

\begin{prop}\label{prop:16k8-odd-16m8} Let
$$N=\begin{pmatrix}16k+8&2m+1\\0&16n+8\end{pmatrix}\in UT_2(\Z). $$ 
Then there are no four matrices in $UT_2(\mathbb Z)$ having the Jordan $D(N)$-property.
Consequently, there is no upper-triangular Jordan $D(N)$-quadruple.
\end{prop}
\begin{proof}
Since
$$ a_ia_j+8\equiv \square,~  d_id_j+8 \equiv \square \pmod{16},$$
    the only possible residue patterns modulo 4 for $a_i$'s and $d_i$'s, up to permutation, are 
$$ (1,1,1,1), ~ (3,3,3,3), ~ (0,1,1,1), ~ (0,3,3,3).$$
Thus, among the numbers $a_1,\ldots,a_4$, at most one is even, and the same holds for $d_1,\ldots,d_4$. Consequently, exactly the same four parity cases arise as in Proposition~\ref{prop:4k3-4n3-odd}. From this point on, the argument depends only on the integrality condition, the square condition for the $(1,2)$-entries of the Jordan products, and the fact that the $(1,2)$-entry of $N$ is odd. For pairs of odd diagonal entries, the corresponding diagonal square roots are both odd, so their sum is even. Hence, the square-root trace is even, as required in the proof of  Proposition~\ref{prop:4k3-4n3-odd}, and the same parity contradictions are obtained. 
\end{proof}

\subsection{Summary and remaining cases}
\begin{center}
\begin{table}[h!]    
\begin{tabular}{c|c|c}
Case & Status & Remaining open cases\\ \hline \hline
I & completely solved & -- \\ \hline

II & partially solved &
 $(p,q)\equiv(1,1)\pmod8$, $r\equiv1\pmod2$ \\ \hline

III & partially solved &
$(p,q)\equiv(4,4)\pmod{16}$,
$r\equiv1,2,3\pmod4$ \\ \hline

IV & completely solved & --
\end{tabular}
\caption{Status of the  non-representable congruence cases}
    \label{tab1}
\end{table}\end{center}

The entries in the last column indicate congruence classes that are not completely classified, although some infinite existence subfamilies within these classes are known. Proposition~\ref{prop:s2-m-t2} shows, in particular, that the square-diagonal subfamilies of both residual classes are completely settled positively. Further infinite families of Jordan $D(N)$-quadruples in these residual classes are provided by Propositions~\ref{prop:quadruple-1m1} and~\ref{prop:4-2t-4}.

To summarize, within Cases II and III, both existence and non-existence phenomena occur. However, the residual congruence classes listed in the table remain open.

\section{Polynomial constructions of Jordan $D(N)$-quadruples}

Although the existence of such quadruples already follows from Corollary~\ref{cor:ex-of-quadruples}, we include the following direct construction to illustrate how classical polynomial formulas from \cite{duje-graz96} for Diophantine quadruples can be adapted to the Jordan setting. The main difference is that, besides matching the diagonal entries, one has to choose the upper-right entries so that the corresponding square conditions are satisfied. In the example below, this requirement reduces to the solvability of a linear Diophantine equation.

We give an explicit construction for 
$$N=\begin{pmatrix}
    4k+3& 2m\\ 0& 4n+3
\end{pmatrix}.$$
We seek a Jordan $D(N)$-pair of the form 
    $$ \left\{A=\begin{pmatrix}
        1& b\\ 0& 1 
    \end{pmatrix},B=\begin{pmatrix}
        9k^2+8k+1 & y\\ 0& 9n^2+8n+1
    \end{pmatrix}\right\}.$$
Note that the diagonal entries belong to the polynomial quadruple  \eqref{poli.formula} (for $\ell=1$). Following the construction of Dujella \cite{duje-graz96}, we extend this pair to a family with the Jordan $D(N)$-property by setting
    $$ C=A+B+2R, $$
    where
    $$ A\circ B+N=\begin{pmatrix}
        (3k+2)^2&\frac{b}{2}(2+8k+9k^2+8n+9n^2)+y+2m\\0&(3n+2)^2
    \end{pmatrix}.$$
In order that $R=\begin{pmatrix}
    3k+2&w\\0&3n+2
\end{pmatrix}$ satisfy $R^2
=A\circ B+N$, its upper-right entry $w$ must satisfy
$$ -\frac{b}{2}(2+8k+9k^2+8n+9n^2)-y+w(4+3k+3n)=2m.$$
Thus, 
\begin{eqnarray*}
    y&=&w(4+3k+3n)-\frac{b}{2}(2+8k+9k^2+8n+9n^2)-2m,
    \end{eqnarray*}
$w,b\in\Z$. If $k+n\equiv 1\pmod2$, then $b$ has to be even. We now extend the pair $\{B,C\}$ by setting
$$ D=A+4B+4R.$$
Hence, $(A,B,C,D)$ has the Jordan $D(N)$-property if and only if
$$ A\circ D+N=\square.$$
We have

  $$\frac12(AD+DA)+N=\begin{pmatrix}
    4 (2 + 3 k)^2&
 b(14  + 22  k + 18  k^2  + 22  n + 18  n^2) + 4 w + 4 y+2m\\0&
 4 (2 + 3 n)^2
\end{pmatrix}.$$
This matrix is the square of
$$\begin{pmatrix}
    2 (2 + 3 k)&
 z\\0&
 2(2 + 3 n)
\end{pmatrix} $$
if and only if the Diophantine equation
$$ b(5+3k+3n)+2w(5+3k+3n)-z(4+3k+3n)=3m$$
is solvable in $b,w,z$. 
Since $$\gcd(5+3k+3n,4+3k+3n)=1,$$
the equation is solvable. 
A general solution of the previous Diophantine equation is
\begin{eqnarray*}
    z&=&(5+3k+3n)t+3m\\
    b&=&(4+3k+3n)t+3m-2w,
\end{eqnarray*}
where $t,w\in\Z$. Note that  $b$ has to be even if  $k+n\equiv1\pmod 2$, but this can be achieved by choosing $t\equiv m\pmod2$. Hence there are infinitely many integer pairs $(t,w)$ satisfying the required parity condition.

It remains to ensure that the matrices $A,B,C,D$ are pairwise distinct.
Their diagonal entries are obtained from the four polynomials
$$f_1(x)=1,\ 
f_2(x)=9x^2+8x+1,\ 
f_3(x)=9x^2+14x+6,\ 
f_4(x)=36x^2+44x+13.$$
For integer $x$, the only equalities among these polynomials are
$$f_1(x)=f_2(x)\iff x=0,
\qquad
f_1(x)=f_3(x)\iff x=-1.$$
Consequently, equality between two of the matrices can occur only in the
cases $A=B$ with $k=n=0$, or $A=C$ with $k=n=-1$.

If $k=n=0$, then
\[A=B\iff b=y\iff w=t+m,\]
whereas, if $k=n=-1$, then
\[A=C\iff y+2w=0\iff t+w=2m.\]
Thus, in each exceptional case, equality occurs only on one affine
line in the $(t,w)$-parameter space. Avoiding this line still leaves
infinitely many integer pairs $(t,w)$. Furthermore, all four matrices
are nonzero, since none of the polynomials $f_1,f_2,f_3,f_4$ has an
integer zero. Therefore, $t$ and $w$ can be chosen so that
$A,B,C,D$ are pairwise distinct and nonzero. Hence they form a
Jordan $D(N)$-quadruple.

\medskip

Together with the prototype construction from Theorem~\ref{thm:quadruple_for_diff_sq}, this demonstrates two complementary approaches to constructing Jordan D(N)-quadru-ples in $UT_2(\Z)$: one based on representations as differences of two squares, and one based on adaptations of classical polynomial formulas. The same method {\color{black}may also be applicable} to other polynomial families of classical $D(n)$-quadruples.

\section{Acknowledgments}
The authors acknowledge support from:
\begin{itemize}
    \item[-] the Croatian Science Foundation under the
project no. IP-2022-10-5008 (TEBAG),
\item[-] the
project “Implementation of cutting-edge research and its application as part of the
Scientific Center of Excellence for Quantum and Complex Systems, and Representations of Lie Algebras”, Grant No. PK.1.1.10.0004, co-financed by the European
Union through the European Regional Development Fund – Competitiveness and
Cohesion Programme 2021--2027,
\item[-] the European Union:
NextGenerationEU through the National Recovery and Resilience Plan 2021--2026, Institutional grant of University of Zagreb Faculty of Science (IK IA 1.1.3. Impact4Math).
\end{itemize}

The authors acknowledge the use of ChatGPT (OpenAI) during the early stages of this work. In particular, discussions with ChatGPT helped inspire the initial idea leading to the construction presented in Theorem~\ref{thm:quadruple_for_diff_sq}. The mathematical development, proofs, verification of results, and final presentation are the authors' own.

The authors would like to thank Professors  Tomislav Pejkovi\' c and Matija Kazalicki for carefully reading the manuscript and for their useful comments and suggestions.

\end{document}